\documentclass[]{article}

\usepackage{graphicx}
\usepackage{amsmath}
\usepackage{amssymb}
\usepackage{amsthm}
\usepackage{pxfonts}
\usepackage{enumerate}
\usepackage{color}
\usepackage{mathdots}
\usepackage{sectsty}
\usepackage[hidelinks]{hyperref}

\sectionfont{\scshape\centering\fontsize{11}{14}\selectfont}
\subsectionfont{\scshape\fontsize{11}{14}\selectfont}
\usepackage{fancyhdr}

\newcommand\shorttitle{Intertwinings for general $\beta$-Laguerre and $\beta$-Jacobi processes}
\newcommand\authors{T. Assiotis}

\newtheorem{thm}{Theorem}[section]
\newtheorem{cor}[thm]{Corollary}

\newtheorem{rmk}[thm]{Remark}

\title{\large \bf INTERTWININGS FOR GENERAL $\beta$-LAGUERRE AND $\beta$-JACOBI PROCESSES}
\author{\small THEODOROS ASSIOTIS}
\date{}

\begin{document}

\maketitle

\begin{abstract}
We show that, for $\beta \ge 1$, the semigroups of $\beta$-Laguerre and $\beta$-Jacobi processes of different dimensions are intertwined in analogy to a similar result for $\beta$-Dyson Brownian motion recently obtained in \cite{RamananShkolnikov}. These intertwining relations generalize to arbitrary $\beta \ge 1$ the ones obtained for $\beta=2$ in \cite{InterlacingDiffusions} between $h$-transformed Karlin-McGregor semigroups. Moreover, they form the key step towards constructing a multilevel process in a Gelfand-Tsetlin pattern leaving certain Gibbs measures invariant. Finally, as a by product, we obtain a relation between general $\beta$-Jacobi ensembles of different dimensions.
\end{abstract}

\section{Introduction}
The aim of this short note is to establish intertwining relations between the semigroups of general $\beta$-Laguerre and $\beta$-Jacobi processes, in analogy to the ones obtained for general $\beta$-Dyson Brownian motion in \cite{RamananShkolnikov} (see also \cite{GorinShkolnikov}). These, also generalize the relations obtained for $\beta=2$  in \cite{InterlacingDiffusions} when the transition kernels for these semigroups are given explicitly in terms of $h$-transforms of Karlin-McGregor determinants.

We begin, by introducing the stochastic processes we will be dealing with. Consider the unique strong solution to the following system of $SDEs$ with $i=1,\cdots,n$ with values in $[0,\infty)^n$,
\begin{align}\label{BESQsde}
dX_i^{(n)}(t)=2\sqrt{X_i^{(n)}(t)}dB_i^{(n)}(t)+\beta\left(\frac{d}{2}+\sum_{1\le j \le k, j \ne i}^{}\frac{2X_i^{(n)}(t)}{X_i^{(n)}(t)-X_j^{(n)}(t)}\right)dt,
\end{align}
where the $B_i^{(n)}$, $i=1,\cdots, n,$ are independent standard Brownian motions. This process, was introduced and studied by Demni in \cite{DemniBESQ} in relation to Dunkl processes, (see for example \cite{DunklProcesses}) where it is referred to as the $\beta$-Laguerre process, since its distribution at time $1$, if started from the origin, is given by the $\beta$-Laguerre ensemble (see Section 5 of \cite{DemniBESQ}). We could, equally well, have called this the $\beta$-squared Bessel process, since for $\beta=2$ it exactly consists of $n$ $BESQ(d)$ diffusion processes conditioned to never collide as first proven in \cite{O Connell} but we stick to the terminology of \cite{DemniBESQ}. Similarly, consider the unique strong solution to the following system of $SDEs$ in $[0,1]^n$,
\begin{align}\label{Jacobisde}
dX_i^{(n)}(t)=2\sqrt{X_i^{(n)}(t)(1-X_i^{(n)}(t))}dB_i^{(n)}(t)+\beta\left(a-(a+b)X_i^{(n)}(t)+\sum_{1\le j \le k, j \ne i}^{}\frac{2X_i^{(n)}(t)(1-X_i^{(n)}(t))}{X_i^{(n)}(t)-X_j^{(n)}(t)}\right)dt,
\end{align}
where, again, the $B_i^{(n)}$, $i=1,\cdots, n,$ are independent standard Brownian motions. We call this solution the $\beta$-Jacobi process. It was first introduced and studied in \cite{DemniJacobi} as a generalization of the eigenvalue evolutions of matrix Jacobi processes and whose stationary distribution is given by the $\beta$-Jacobi ensemble (see Section 4 of \cite{DemniJacobi}):
\begin{align}
\mathcal{M}^{Jac,n}_{a,b,\beta}(dx)=C_{n,a,b,\beta}^{-1}\prod_{i=1}^{n}x^{\frac{\beta}{2}a-1}_{i}(1-x_i)^{\frac{\beta}{2}b-1}\prod_{1\le i < j \le n}^{}|x_j-x_i|^{\beta}dx,
\end{align}
for some normalization constant $C_{n,a,b,\beta}$.

We now give sufficient conditions that guarantee the well-posedness of  the $SDEs$ above. For $\beta \ge 1$ and $d\ge 0$ and $a,b\ge 0$, (\ref{BESQsde}) and (\ref{Jacobisde}) have a unique strong solution with no collisions and no explosions and with instant diffraction if started from a degenerate (i.e. when some of the coordinates coincide) point (see Corollary 6.5 and 6.7 respectively of \cite{Graczyk}). In particular, the coordinates of $X^{(n)}$ stay ordered. Thus if,
\begin{align*}
X^{(n)}_1(0) \le \cdots \le X^{(n)}_n(0),
\end{align*}
then with probability one,
\begin{align*}
X^{(n)}_1(t) < \cdots < X^{(n)}_n(t), \ \forall  \ t>0.
\end{align*}
From now on, we restrict to those parameter values.

It will be convenient to define $\theta=\frac{\beta}{2}$. We write $P^{(n)}_{d,\theta}(t)$ for the Markov semigroup associated to the solution of (\ref{BESQsde}). Similarly, write $Q^{(n)}_{a,b,\theta}(t)$ for the Markov semigroup associated to the solution of (\ref{Jacobisde}).
Furthermore, denote by $\mathcal{L}^{(n)}_{d,\theta}$ and $\mathcal{A}^{(n)}_{a,b,\theta}$ the formal infinitesimal generators for (\ref{BESQsde}) and (\ref{Jacobisde}) respectively, given by,
\begin{align}
\mathcal{L}^{(n)}_{d,\theta}&=\sum_{i=1}^{n}2z_i\frac{\partial}{\partial z^2_i}+2 \theta \sum_{i=1}^{n}\left(\frac{d}{2}+\sum_{1\le j \le k, j \ne i}^{}\frac{2z_i}{z_i-z_j}\right)\frac{\partial}{\partial z_i},\\
\mathcal{A}^{(n)}_{a,b,\theta}&=\sum_{i=1}^{n}2z_i(1-z_i)\frac{\partial}{\partial z^2_i}+2 \theta \sum_{i=1}^{n}\left(a-(a+b)z_i+\sum_{1\le j \le k, j \ne i}^{}\frac{2z_i(1-z_i)}{z_i-z_j}\right)\frac{\partial}{\partial z_i}.
\end{align}
With $I$ denoting either $[0,\infty)$ or $[0,1]$, define the chamber,
\begin{align*}
W^n(I)=\{x=(x_1,\cdots,x_n)\in I^n:x_1\le \cdots \le x_n\}.
\end{align*}
Moreover, for $x\in W^{n+1}$ define the set of $y \in W^{n}$ that \textit{interlace} with $x$ by,
\begin{align*}
W^{n,n+1}(x)=\{y=(y_1,\cdots,y_n)\in I^n: x_1\le y_1 \le x_2 \le \cdots \le y_n \le x_{n+1}\}.
\end{align*}
For $x\in W^{n+1}$ and $y\in W^{n,n+1}(x)$, define the \textit{Dixon-Anderson} conditional \textit{probability} density on $W^{n,n+1}(x)$ (originally introduced by Dixon at the beginning of the last century in \cite{Dixon} and independently rediscovered by Anderson in his study of the Selberg integral in \cite{Anderson}) by,
\begin{align}
\lambda^{\theta}_{n,n+1}(x,y)=\frac{\Gamma (\theta (n+1))}{\Gamma(\theta)^{n+1}}\prod_{1\le i <j \le n+1}^{}(x_j-x_i)^{1-2\theta}\prod_{1\le i <j \le n}^{}(y_j-y_i)\prod_{i=1}^{n}\prod_{j=1}^{n+1}|y_i-x_j|^{\theta-1}.
\end{align}
Denote by $\Lambda^{\theta}_{n,n+1}$, the integral operator with kernel $\lambda^{\theta}_{n,n+1}$ i.e.,
\begin{align*}
(\Lambda^{\theta}_{n,n+1}f)(x)=\int_{y\in W^{n,n+1}(x)}^{}\lambda^{\theta}_{n,n+1}(x,y)f(y)dy.
\end{align*}
Then, our goal is to prove the following theorem, which should be considered as a generalization to the other two classical $\beta$-ensembles, the $Laguerre$ and $Jacobi$, of the result of \cite{RamananShkolnikov} for the $Gaussian$ ensemble.
\begin{thm}\label{MainTheorem}
Let $\beta \ge 1$, $d\ge 2$ and $a,b \ge 1$. Then, with $\theta=\frac{\beta}{2}$, we have the following equalities of Markov kernels, $\forall t \ge 0$,
\begin{align}
P^{(n+1)}_{d-2,\theta}(t)\Lambda^{\theta}_{n,n+1}&=\Lambda^{\theta}_{n,n+1}P^{(n)}_{d,\theta}(t)\label{BESQintertwining},\\
Q^{(n+1)}_{a-1,b-1,\theta}(t)\Lambda^{\theta}_{n,n+1}&=\Lambda^{\theta}_{n,n+1}Q^{(n)}_{a,b,\theta}(t) \label{Jacobiintertwining}.
\end{align}
\end{thm}
\begin{rmk}
For $\beta=2$, this result was already obtained in \cite{InterlacingDiffusions}, see in particular subsections 3.7 and 3.8 therein respectively.
\end{rmk} 

\begin{rmk}
The general theory of intertwining diffusions (see \cite{PalShkolnikov}), suggests that there should be a way to realize these intertwining relations by coupling these $n$ and $n+1$ particle processes, so that they interlace. In the Laguerre case, (the Jacobi case is analogous) the resulting process $Z=(X,Y)$, with $Y$ evolving according to $P^{(n)}_{d,\theta}(t)$ and $X$ in its own filtration according to $P^{(n+1)}_{d-2,\theta}(t)$, should (conjecturally) have generator given by,
\begin{align*}
\mathcal{L}^{n,n+1}_{\beta,d}=\sum_{j=1}^{n}2y_j\partial^2_{y_j}+\beta\sum_{j=1}^{n}\left(\frac{d}{2}+\sum_{i\ne j}^{}\frac{2y_j}{y_j-y_i}\right)\partial_{y_j}+\sum_{j=1}^{n+1}2x_j\partial^2_{x_j}+\beta\sum_{j=1}^{n+1}\left(\frac{d-2}{2}+\sum_{i\ne j}^{}\frac{2x_j}{x_j-x_i}\right)\partial_{x_j}\\
+(1-\beta)\sum_{j=1}^{n+1}\sum_{i \ne j}^{}\frac{4x_j}{x_i-x_j}\partial_{x_j}+\left(\frac{\beta}{2}-1\right)\sum_{j=1}^{n+1}\sum_{i=1}^{n}\frac{4x_j}{x_j-y_i}\partial_{x_j},
\end{align*}
with reflecting boundary conditions of the $X$ components on the $Y$ particles (in case they do collide). For a rigorous construction of the analogous coupled process in the case of Dyson Brownian motions with $\beta>2$, see Section 4 of \cite{GorinShkolnikov}. In fact, for certain values of the parameters, the construction of the process with the generator above, can be reduced to the results of $\cite{GorinShkolnikov}$ and a more detailed account will appear as part of the author's $PhD$ thesis \cite{PhDThesis}.

As just mentioned, such a coupling was constructed for Dyson Brownian motion with $\beta > 2$ in \cite{GorinShkolnikov}; and in \cite{InterlacingDiffusions} (see also \cite{Sun}) for copies of general one-dimensional diffusion processes, that in particular includes the squared Bessel (this corresponds to the $Laguerre$ process of this note) and Jacobi cases for $\beta=2$, when the interaction, between the two levels, entirely consists of local hard reflection and the transition kernels are explicit. Given such 2-level couplings, one can then iterate to construct a multilevel process in a Gelfand-Tsetlin pattern, as in \cite{Warren} which initiated this program (see also \cite{GorinShkolnikov},\cite{PalShkolnikov},\cite{InterlacingDiffusions}). For a different type of coupling, for $\beta=2$ Dyson Brownian motion, that preceded \cite{O Connell} and is related to the Robinson-Schensted correspondence, see \cite{O ConnellTams}, \cite{O ConnellYor} and the related work \cite{BougerolJeulin}.
\end{rmk}

Using Theorem \ref{MainTheorem} and that $\mathcal{M}^{Jac,n}_{a,b,\beta}$ is the \textit{unique} stationary measure of (\ref{Jacobisde}) which follows from smoothness and positivity of the transition density $p^{n,\beta,a,b}_t(x,y)$, with respect to Lebesgue measure of $Q^{(n)}_{a,b,\theta}(t)$ (see Proposition 4.1 of \cite{DemniJacobi}; for this to apply we further need to restrict to $a,b > \frac{1}{\beta}$) and the fact that two distinct ergodic measures must be mutually singular (see \cite{Walters}), we immediately get:
\begin{cor}
For $\beta\ge 1$ and $a,b > 1$ and with $\theta=\frac{\beta}{2}$,
\begin{align}
\mathcal{M}^{Jac,n+1}_{a-1,b-1,\beta}\Lambda^{\theta}_{n,n+1}=\mathcal{M}^{Jac,n}_{a,b,\beta}.
\end{align}
\end{cor}
\begin{proof}
From (\ref{Jacobiintertwining}) we obtain that $\mathcal{M}^{Jac,n+1}_{a-1,b-1,\beta}\Lambda^{\theta}_{n,n+1}$ is the unique stationary measure of $Q^{(n)}_{a,b,\theta}(t)$
\end{proof}
Before closing this introduction we remark, that in order to establish Theorem \ref{MainTheorem}, we will follow the strategy given in \cite{RamananShkolnikov}, namely we rely on the explicit action of the generators and integral kernel on the class of Jack polynomials which, along with an exponential moment estimate, will allow us to apply the moment method. We note that, although the $\beta$-Laguerre and $\beta$-Jacobi diffusions look more complicated than $\beta$-Dyson's Brownian motion, the main computation, performed in Step 1 of the proof below, is actually simpler than the one in \cite{RamananShkolnikov}.

\paragraph{Acknowledgements} I would like to thank Jon Warren for several useful comments on an earlier draft of this note and also Neil O'Connell and Nizar Demni for some historical remarks. Finally, I would like to thank an anonymous referee for detailed comments and suggestions that led to an improvement of the exposition. Financial support through the MASDOC DTC grant number EP/HO23364/1 is gratefully acknowledged.

\section{Preliminaries on Jack polynomials}
We collect some facts on the Jack polynomials $J_{\lambda}(z;\theta)$ which as already mentioned will play a key role in obtaining these intertwining relations. We mainly follow \cite{RamananShkolnikov} which in turn follows \cite{BakerForrester} (note that there is a misprint in \cite{RamananShkolnikov}; there is a factor of $\frac{1}{2}$ missing from equation (2.7) therein c.f. equation (2.13d) in \cite{BakerForrester}). The $J_{\lambda}(z;\theta)$ are defined to be the (unique up to normalization) symmetric polynomial eigenfunctions in $n$ variables of the differential operator $\mathcal{D}^{(n),\theta}$,
\begin{align}
\mathcal{D}^{(n),\theta}=\sum_{i=1}^{n}z^2_i\frac{\partial}{\partial z^2_i}+2 \theta \sum_{i=1}^{n}\sum_{1\le j \le k, j \ne i}^{}\frac{z^2_i}{z_i-z_j}\frac{\partial}{\partial z_i},
\end{align}
indexed by partitions $\lambda=(\lambda_1 \ge \lambda_2\ge \cdots)$ of length $l$ with eigenvalue $eval(\lambda,n,\theta)=2B(\lambda')-2\theta B(\lambda)+2\theta(n-1)|\lambda|$ where $B(\lambda)=\sum(i-1)\lambda_i=\sum \binom{\lambda'_i}{2}$ and $\lambda'$ is the conjugate partition. With $1_n$ denoting a row vector of $n$ $1$s, we have the normalization,
\begin{align*}
J_{\lambda}(1_n;\theta)=\theta^{-|\lambda|}\prod_{i=1}^{l}\frac{\Gamma\left(\left(n+1-i\right)\theta+\lambda_i\right)}{\Gamma\left(\left(n+1-i\right)\theta\right)}.
\end{align*}
Define the following differential operators,
\begin{align}
\mathcal{B}_1^{(n)}&= \sum_{i=1}^{n}\frac{\partial}{\partial z_i},\\
\mathcal{B}_2^{(n),\theta}&=\sum_{i=1}^{n}z_i\frac{\partial}{\partial z^2_i}+2 \theta \sum_{i=1}^{n}\sum_{1\le j \le k, j \ne i}^{}\frac{z_i}{z_i-z_j}\frac{\partial}{\partial z_i},\\
\mathcal{B}_3^{(n)}&= \sum_{i=1}^{n}z_i\frac{\partial}{\partial z_i}.
\end{align}
Then the action of these operators on the $J_{\lambda}(z;\theta)$'s is given explicitly by (see \cite{BakerForrester} equations $(2.13a)$, $(2.13d)$ and $(2.13b)$ respectively),
\begin{align}
\mathcal{B}_1^{(n)}J_{\lambda}(z;\theta)&=J_{\lambda}(1_n;\theta)\sum_{i=1}^{l}\binom{\lambda}{\lambda_{(i)}}_{\theta}\frac{J_{\lambda_{(i)}}(z;\theta)}{J_{\lambda_{(i)}}(1_n;\theta)},\\
\mathcal{B}_2^{(n),\theta}J_{\lambda}(z;\theta)&=J_{\lambda}(1_n;\theta)\sum_{i=1}^{l}\binom{\lambda}{\lambda_{(i)}}_{\theta}(\lambda_i-1+(n-i)\theta)\frac{J_{\lambda_{(i)}}(z;\theta)}{J_{\lambda_{(i)}}(1_n;\theta)},\\
\mathcal{B}_3^{(n)}J_{\lambda}(z;\theta)&=|\lambda|J_{\lambda}(z;\theta),
\end{align}
where $\lambda_{(i)}$ is the sequence given by $\lambda_{(i)}=(\lambda_1,\cdots,\lambda_{i-1},\lambda_i-1,\lambda_{i+1},\cdots)$ (in case $i=l$ and $\lambda_i=1$ we drop $\lambda_l$ from $\lambda$) and the combinatorial coefficients $\binom{\lambda}{\rho}_{\theta}$ are defined by the following expansion (we set $\binom{\lambda}{\lambda_{(i)}}_{\theta}=0$ in case $\lambda_{(i)}$ is no longer a non-decreasing positive sequence),
\begin{align*}
\frac{J_{\lambda}(1_n+z;\theta)}{J_{\lambda}(1_n;\theta)}=\sum_{m=0}^{|\lambda|}\sum_{|\rho|=m}^{}\binom{\lambda}{\rho}_{\theta}\frac{J_{\rho}(z;\theta)}{J_{\rho}(1_n;\theta)},
\end{align*}
but whose exact values will not be required in what follows. Finally, we need the following about the action of $\Lambda^{\theta}_{n,n+1}$ on $J_{\lambda}(\cdot;\theta)$ (see \cite{OkounkovOlshanski} Section 6),
\begin{align}\label{kernelonjack}
\int_{W^{n,n+1}(x)}^{}\lambda^{\theta}_{n,n+1}(x,y)J_{\lambda}(y;\theta)dy=J_{\lambda}(x;\theta)c(\lambda,n,\theta) ,
\end{align}
where,
\begin{align}
c(\lambda,n,\theta)=\frac{\Gamma((n+1)\theta)}{\Gamma(\theta)}\prod_{i=1}^{n}\frac{\Gamma\left(\left(n+1-i\right)\theta+\lambda_i\right)}{\Gamma\left(\left(n+2-i\right)\theta+\lambda_i\right)}.
\end{align}

\section{Proof}
We split the proof in 4 steps, following the strategy laid out in \cite{RamananShkolnikov}.
\begin{proof}[Proof of Theorem \ref*{MainTheorem}]
First, note that we can write the operators $\mathcal{L}^{(n)}_{d,\theta}$ and $\mathcal{A}^{(n)}_{a,b,\theta}$ as follows,
\begin{align}
\mathcal{L}^{(n)}_{d,\theta}&=2\mathcal{B}_2^{(n),\theta}+\theta d \mathcal{B}_1^{(n)},\\
\mathcal{A}^{(n)}_{a,b,\theta}&=2\mathcal{B}_2^{(n),\theta}-2\mathcal{D}^{(n),\theta}+2\theta a \mathcal{B}_1^{(n)}-2\theta(a+b)\mathcal{B}_3^{(n),\theta}.
\end{align}
\paragraph{Step 1} The aim of this step is to show the intertwining relation at the level of the infinitesimal generators acting on the Jack polynomials. Namely that,
\begin{align}
\mathcal{L}^{(n+1)}_{d-2,\theta}\Lambda^{\theta}_{n,n+1}J_{\lambda}(\cdot;\theta)&=\Lambda^{\theta}_{n,n+1}\mathcal{L}^{(n)}_{d,\theta}J_{\lambda}(\cdot;\theta)\label{BESQgeneratorintertwining}, \\
\mathcal{A}^{(n+1)}_{a-1,b-1,\theta}\Lambda^{\theta}_{n,n+1}J_{\lambda}(\cdot;\theta)&=\Lambda^{\theta}_{n,n+1}\mathcal{A}^{(n)}_{a,b,\theta}J_{\lambda}(\cdot;\theta)\label{Jacgeneratorintertwining}.
\end{align}
We will show relation (\ref*{Jacgeneratorintertwining}) for the Jacobi case and at the end of Step 1 indicate how to obtain (\ref*{BESQgeneratorintertwining}).

\textbf{(LHS)}=
\begin{align*}
&\mathcal{A}^{(n+1)}_{a-1,b-1,\theta}J_{\lambda}(x;\theta)c(\lambda,n,\theta)=c(\lambda,n,\theta)\left(2\mathcal{B}_2^{(n+1),\theta}-2\mathcal{D}^{(n+1),\theta}+2\theta (a-1) \mathcal{B}_1^{(n+1)}-2\theta(a+b-2)\mathcal{B}_3^{(n+1),\theta}\right)J_{\lambda}(x;\theta)\\
&=c(\lambda,n,\theta)\bigg[2J_{\lambda}(1_{n+1};\theta)\sum_{i=1}^{l}\binom{\lambda}{\lambda_{(i)}}_{\theta}(\lambda_i-1+(n+1-i)\theta)\frac{J_{\lambda_{(i)}}(x;\theta)}{J_{\lambda_{(i)}}(1_{n+1};\theta)}-2eval(\lambda,n+1,\theta)J_{\lambda}(x;\theta)\\ &+2\theta(a-1)J_{\lambda}(1_{n+1};\theta)\sum_{i=1}^{l}\binom{\lambda}{\lambda_{(i)}}_{\theta}\frac{J_{\lambda_{(i)}}(x;\theta)}{J_{\lambda_{(i)}}(1_{n+1};\theta)}-2\theta(a+b-2)|\lambda|J_{\lambda}(x;\theta)\bigg].
\end{align*} 

\textbf{(RHS)}: We start by computing $\mathcal{A}^{(n)}_{a,b,\theta}J_{\lambda}(y;\theta)$.
\begin{align}
&\mathcal{A}^{(n)}_{a,b,\theta}J_{\lambda}(y;\theta)=\left(2\mathcal{B}_2^{(n),\theta}-2\mathcal{D}^{(n),\theta}+2\theta a \mathcal{B}_1^{(n)}-2\theta(a+b)\mathcal{B}_3^{(n),\theta}\right)J_{\lambda}(y;\theta) \nonumber\\
&=\bigg[2J_{\lambda}(1_{n};\theta)\sum_{i=1}^{l}\binom{\lambda}{\lambda_{(i)}}_{\theta}(\lambda_i-1+(n-i)\theta)\frac{J_{\lambda_{(i)}}(y;\theta)}{J_{\lambda_{(i)}}(1_{n+1};\theta)}-2eval(\lambda,n,\theta)J_{\lambda}(y;\theta)\label{linearcombination}\\ &+2\theta a J_{\lambda}(1_{n};\theta)\sum_{i=1}^{l}\binom{\lambda}{\lambda_{(i)}}_{\theta}\frac{J_{\lambda_{(i)}}(y;\theta)}{J_{\lambda_{(i)}}(1_{n};\theta)}-2\theta(a+b)|\lambda|J_{\lambda}(y;\theta)\bigg] \nonumber.
\end{align} 
Now, apply $\Lambda^{\theta}_{n,n+1}$ to obtain that,
\begin{align*}
\textbf{(RHS)}&=2J_{\lambda}(1_{n};\theta)\sum_{i=1}^{l}\binom{\lambda}{\lambda_{(i)}}_{\theta}(\lambda_i-1+(n-i)\theta)c(\lambda_{(i)},n,\theta)\frac{J_{\lambda_{(i)}}(x;\theta)}{J_{\lambda_{(i)}}(1_{n+1};\theta)}-2c(\lambda,n,\theta)eval(\lambda,n,\theta)J_{\lambda}(x;\theta)\\ &
+2\theta a J_{\lambda}(1_{n};\theta)\sum_{i=1}^{l}\binom{\lambda}{\lambda_{(i)}}_{\theta}c(\lambda_{(i)},n,\theta)\frac{J_{\lambda_{(i)}}(x;\theta)}{J_{\lambda_{(i)}}(1_{n};\theta)}-2\theta(a+b)|\lambda|c(\lambda,n,\theta)J_{\lambda}(x;\theta).
\end{align*}

Now, in order to check \textbf{(LHS)}=\textbf{(RHS)} we check that the coefficients of $J_{\lambda}$ and $J_{\lambda_{(i)}}$ $\forall i$ coincide on both sides.\\

$\bullet$ First, the coefficients of $J_{\lambda}(x;\theta)$:

\textbf{(LHS)}: $-2c(\lambda,n,\theta)eval(\lambda,n+1,\theta)-c(\lambda,n,\theta)|\lambda|2 \theta (a+b-2)$.

\textbf{(RHS)}: $-2c(\lambda,n,\theta)eval(\lambda,n,\theta)-c(\lambda,n,\theta)|\lambda|2 \theta (a+b)$.

These are equal iff:
\begin{align*}
\frac{-2eval(\lambda,n,\theta)+2eval(\lambda,n+1,\theta)}{4\theta |\lambda|}=1,
\end{align*}
which is easily checked from the explicit expression of $eval(n,\lambda,\theta)$.\\

$\bullet$ Now, for the coefficients of $J_{\lambda_{(i)}}(x;\theta)$:

\textbf{(LHS)}:
\begin{align*}
& 2J_{\lambda}(1_{n+1};\theta)\binom{\lambda}{\lambda_{(i)}}_{\theta}(\lambda_i-1+(n+1-i)\theta)\frac{c(\lambda,n,\theta)}{J_{\lambda_{(i)}}(1_{n+1};\theta)}+2\theta(a-1)J_{\lambda}(1_{n+1};\theta)\binom{\lambda}{\lambda_{(i)}}_{\theta}\frac{c(\lambda,n,\theta)}{J_{\lambda_{(i)}}(1_{n+1};\theta)}.
\end{align*} 

\textbf{(RHS)}: 
\begin{align*}
& 2J_{\lambda}(1_{n};\theta)\binom{\lambda}{\lambda_{(i)}}_{\theta}(\lambda_i-1+(n-i)\theta)\frac{c(\lambda_{(i)},n,\theta)}{J_{\lambda_{(i)}}(1_{n};\theta)}+2\theta aJ_{\lambda}(1_{n};\theta)\binom{\lambda}{\lambda_{(i)}}_{\theta}\frac{c(\lambda_{(i)},n,\theta)}{J_{\lambda_{(i)}}(1_{n};\theta)}.
\end{align*} 
These are equal iff:
\begin{align*}
a-1=\frac{J_{\lambda}(1_{n};\theta)c(\lambda_{(i)},n,\theta)J_{\lambda_{(i)}}(1_{n+1};\theta)}{J_{\lambda_{(i)}}(1_{n};\theta)c(\lambda,n,\theta)J_{\lambda}(1_{n+1};\theta)}a+\frac{1}{\theta}\frac{J_{\lambda}(1_{n};\theta)c(\lambda_{(i)},n,\theta)J_{\lambda_{(i)}}(1_{n+1};\theta)}{J_{\lambda_{(i)}}(1_{n};\theta)c(\lambda,n,\theta)J_{\lambda}(1_{n+1};\theta)}(\lambda_i-1+(n-i)\theta)\\
-\frac{1}{\theta}(\lambda_i-1+(n+1-i)\theta).
\end{align*}
We first claim that,
\begin{align*}
\frac{J_{\lambda}(1_{n};\theta)c(\lambda_{(i)},n,\theta)J_{\lambda_{(i)}}(1_{n+1};\theta)}{J_{\lambda_{(i)}}(1_{n};\theta)c(\lambda,n,\theta)J_{\lambda}(1_{n+1};\theta)}=1.
\end{align*}
This immediately follows from,
\begin{align*}
\frac{J_{\lambda}(1_{n};\theta)}{J_{\lambda_{(i)}}(1_{n};\theta)}&=\theta^{-1}\frac{\Gamma\left(\left(n+1-i\right)\theta+\lambda_i\right)}{\Gamma\left(\left(n+1-i\right)\theta+\lambda_i-1\right)},\\
\frac{J_{\lambda_{(i)}}(1_{n+1};\theta)}{J_{\lambda}(1_{n+1};\theta)}&=\theta\frac{\Gamma\left(\left(n+2-i\right)\theta+\lambda_i-1\right)}{\Gamma\left(\left(n+2-i\right)\theta+\lambda_i\right)},\\
\frac{c(\lambda_{(i)},n,\theta)}{c(\lambda,n,\theta)}&=\frac{\Gamma\left(\left(n+1-i\right)\theta+\lambda_i-1\right)\Gamma\left(\left(n+2-i\right)\theta+\lambda_i\right)}{\Gamma\left(\left(n+1-i\right)\theta+\lambda_i\right)\Gamma\left(\left(n+2-i\right)\theta+\lambda_i-1\right)}.
\end{align*}
Hence, we need to check that the following is true,
\begin{align*}
a-1=a+\frac{1}{\theta}(\lambda_i-1+(n-i)\theta)-\frac{1}{\theta}(\lambda_i-1+(n-i+1)\theta),
\end{align*}
which is obvious.

Now, in order to obtain (\ref*{BESQgeneratorintertwining}) we only need to consider coefficients in $J_{\lambda_{(i)}}$'s (since the operators $\mathcal{D}^{(n),\theta}$ and $\mathcal{B}_3^{(n)}$ that produce $J_{\lambda}$'s are missing) and replace $a$ by $\frac{d}{2}$.

To prove the analogous result for $\beta$-Dyson Brownian motions, one needs to observe, as done in \cite{RamananShkolnikov}, that the generator of $n$ particle $\beta$-Dyson Brownian motion $L^{(n)}_{\theta}$ can be written as a commutator, namely  $L^{(n)}_{\theta}=[\mathcal{B}_1^{(n)},\mathcal{B}_2^{(n),\theta}]=\mathcal{B}_1^{(n)}\mathcal{B}_2^{(n),\theta}-\mathcal{B}_2^{(n),\theta}\mathcal{B}_1^{(n)}$.

\paragraph{Step 2} We obtain an exponential moment estimate, namely regarding $\mathbb{E}_{x}\left[e^{\epsilon \|X^{(n)}(t)\|}\right]$. This is obviously finite by compactness of $[0,1]^n$ in the Jacobi case. In the Laguerre case, we proceed as follows. Writing $X^{(n)}$ for the solution to (\ref*{BESQsde}), letting $\|\cdot\|$ denote the $l_1$ norm and recalling that all entries of $X^{(n)}$ are non-negative we obtain,
\begin{align*}
d\|X^{(n)}(t)\|=\sum_{i=1}^{n}2\sqrt{dX_i^{(n)}(t)}dB_i^{(n)}(t)+\beta\left(\frac{d}{2}n+\sum_{i=1}^{n}\sum_{1 \le j \le n, j\ne i}^{}\frac{2X_i^{(n)}(t)}{X_i^{(n)}(t)-X_j^{(n)}(t)}\right)dt.
\end{align*}
Note that,
\begin{align*}
\sum_{i=1}^{n}\sum_{1 \le j \le n, j\ne i}^{}\frac{2X_i^{(n)}(t)}{X_i^{(n)}(t)-X_j^{(n)}(t)}=2\binom{n}{2},
\end{align*}
and that by Levy's characterization the local martingale $(M(t),t\ge 0)$ defined by, 
\begin{align*}
dM(t)=\frac{1}{\sqrt{\|X^{(n)}(t)\|}}\sum_{i=1}^{n}\sqrt{X^{(n)}_i(t)}dB^{(n)}_i(t),
\end{align*}
is equal to a standard Brownian motion $(W(t),t\ge 0)$ and so we obtain,
\begin{align*}
d\|X^{(n)}(t)\|=2\sqrt{\|X^{(n)}(t)\|}dW(t)+\beta\left(\frac{d}{2}n+2\binom{n}{2}\right)dt.
\end{align*}
Thus, $\|X^{(n)}(t)\|$ is a squared Bessel process of dimension $dim_{\beta,n,d}=\beta\left(\frac{d}{2}n+2\binom{n}{2}\right)$. Hence, from standard estimates (see \cite{RevuzYor} Chapter IX.1 or Proposition 2.1 of \cite{LDPBessel}; in case that $dim_{\beta,n,d}$ is an integer the result is an immediate consequence of Fernique's theorem (\cite{Fernique}) since $\|X^{(n)}(t)\|$ is the square of a Gaussian process) it follows that, for $\epsilon>0$ small enough, $\mathbb{E}_{x}\left[e^{\epsilon \|X^{(n)}(t)\|}\right]<\infty$.

\paragraph{Step 3} We now lift the intertwining relation to the semigroups acting on the Jack polynomials, namely,
\begin{align*}
P^{(n+1)}_{d-2,\theta}(t)\Lambda^{\theta}_{n,n+1}J_{\lambda}(\cdot;\theta)&=\Lambda^{\theta}_{n,n+1}P^{(n)}_{d,\theta}(t)J_{\lambda}(\cdot;\theta),\\
Q^{(n+1)}_{a-1,b-1,\theta}(t)\Lambda^{\theta}_{n,n+1}J_{\lambda}(\cdot;\theta)&=\Lambda^{\theta}_{n,n+1}Q^{(n)}_{a,b,\theta}(t)J_{\lambda}(\cdot;\theta).
\end{align*}
The proof follows almost word for word the elegant argument given in \cite{RamananShkolnikov}. We reproduce it here, elaborating a bit on some parts, for the convenience of the reader, moreover only considering the Laguerre case for concreteness. We begin by applying Ito's formula to $J_{\lambda}(X^{(n)}(t);\theta)$ and taking expectations (note that the stochastic integral term is a true martingale since its expected quadratic variation is finite which follows by the exponential estimate of Step 2) we obtain,
\begin{align}\label{integral}
P^{(n)}_{d,\theta}(t)J_{\lambda}(\cdot;\theta)=J_{\lambda}(\cdot;\theta)+\int_{0}^{t}P^{(n)}_{d,\theta}(s)\mathcal{L}^{(n)}_{d,\theta}J_{\lambda}(\cdot;\theta)ds.
\end{align}
Now, note that by (\ref{linearcombination}), $\mathcal{L}^{(n)}_{d,\theta}J_{\lambda}(\cdot;\theta)$ is given by a linear combination of Jack polynomials $J_{\kappa}(\cdot;\theta)$ for some partitions $\kappa$ with $\kappa_i\le \lambda_i$ $\forall i \le l$ and we will write $\kappa\le \lambda$ if this holds. We will denote the action of $\mathcal{L}^{(n)}_{d,\theta}$ on this \textit{finite} dimensional vector space, spanned by the Jack polynomials indexed by partitions $\kappa$ with $\kappa \le \lambda$, by the matrix $M_2$. 

Moreover, each $J_{\kappa}(\cdot;\theta)$ for $\kappa \le \lambda$ obeys (\ref{integral}) and thus we obtain the following system of integral equations, with $f_{\kappa}(t)=P^{(n)}_{d,\theta}(t)J_{\kappa}(\cdot;\theta)$,
\begin{align*}
f_{\kappa}(t)=f_{\kappa}(0)+\sum_{\nu \le \lambda}^{}M_2(\kappa,\nu)\int_{0}^{t}f_{\nu}(s)ds,
\end{align*}
whose unique solution is given by the matrix exponential,
\begin{align}\label{matrixexponentialsolution}
f_{\kappa}(t)=\sum_{\nu \le \lambda}^{}e^{tM_2}(\kappa,\nu)f_{\nu}(0).
\end{align}
Now, observe that by (\ref{kernelonjack}) the Markov kernel $\Lambda^{\theta}_{n,n+1}$ also acts on the aforementioned finite dimensional vector space of Jack polynomials as a matrix, which we denote by $M_1$. We will also denote by a matrix $M_3$ the action of $\mathcal{L}^{(n+1)}_{d-2,\theta}$ and note that the intertwining relation (\ref{BESQgeneratorintertwining}) can be written in terms of matrices as follows: $M_3M_1=M_1M_2$. Thus, making use of the following elementary fact about finite dimensional square matrices,
\begin{align*}
M_3M_1=M_1M_2 \implies e^{tM_3}M_1=M_1e^{tM_2} \  \textnormal{for} \ t \ge 0,
\end{align*}
and display (\ref{matrixexponentialsolution}), along with its analogue with $M_2$ replaced by $M_3$, we get that,
\begin{align*}
P^{(n+1)}_{d-2,\theta}(t)\Lambda^{\theta}_{n,n+1}J_{\lambda}(\cdot;\theta)&=\Lambda^{\theta}_{n,n+1}P^{(n)}_{d,\theta}(t)J_{\lambda}(\cdot;\theta).
\end{align*}
\paragraph{Step 4}
We again follow \cite{RamananShkolnikov}. Recall, (see \cite{RamananShkolnikov} and the references therein)  that we can write any \textit{symmetric} polynomial $p$ in $n$ variables as a finite linear combination of Jack polynomials in $n$ variables. Hence, for any such $p$,
\begin{align}
P^{(n+1)}_{d-2,\theta}(t)\Lambda^{\theta}_{n,n+1}p(\cdot)&=\Lambda^{\theta}_{n,n+1}P^{(n)}_{d,\theta}(t)p(\cdot)\label{symmpolyintertwining1},\\
Q^{(n+1)}_{a-1,b-1,\theta}(t)\Lambda^{\theta}_{n,n+1}p(\cdot)&=\Lambda^{\theta}_{n,n+1}Q^{(n)}_{a,b,\theta}(t)p(\cdot)\label{symmpolyintertwining2}.
\end{align}
Now, any probability measure $\mu$ on $W^n(I)$ will give rise to a symmetrized probability measure $\mu^{symm}$ on $I^n$ as follows,
\begin{align*}
\mu^{symm}(dz_1.\cdots,dz_n)=\frac{1}{n!}\mu(dz_{(1)}.\cdots,dz_{(n)}),
\end{align*}
where $z_{(1)}\le z_{(2)}\le \cdots \le z_{(n)}$ are the order statistics of $(z_1,z_2,\cdots,z_n)$. Moreover, for every (not necessarily symmetric) polynomial $q$ in  $n$ variables, with $S_n$ denoting the symmetric group on $n$ symbols, we have,
\begin{align*}
\int_{I^n}^{}q(z)d\mu^{symm}(z)=\int_{I^n}^{}\frac{1}{n!}\sum_{\sigma \in S_n}^{}q(z_{\sigma(1)},\cdots,z_{\sigma(n)})d\mu^{symm}(z)=\int_{W^{n}(I)}^{}\frac{1}{n!}\sum_{\sigma \in S_n}^{}q(z_{\sigma(1)},\cdots,z_{\sigma(n)})d\mu(z).
\end{align*}
Note that now $p(z)=\frac{1}{n!}\sum_{\sigma \in S_n}^{}q(z_{\sigma(1)},\cdots,z_{\sigma(n)})$ is a symmetric polynomial (in $n$ variables). Thus, from (\ref{symmpolyintertwining1}) and (\ref{symmpolyintertwining2}) all moments of the symmetrized versions of both sides of (\ref*{BESQintertwining}) and (\ref{Jacobiintertwining}) coincide. Hence, by Theorem 1.3 of \cite{DeJeu} (and the discussion following it) along with the fact that $(\Lambda^{\theta}_{n,n+1}f)(z)\le e^{\epsilon \|z\|_1}$ where $f(y)=e^{\epsilon \|y\|_1}$ (since all coordinates are positive) and our exponential moment estimate from Step 2 we obtain that the symmetrized versions of both sides of (\ref*{BESQintertwining}) and (\ref{Jacobiintertwining}) coincide; where we view for each $x\in W^{n+1}$ and $t\ge 0$ $P^{(n+1)}_{d-2,\theta}(t)\Lambda^{\theta}_{n,n+1}$ and $\Lambda^{\theta}_{n,n+1}P^{(n)}_{d,\theta}(t)$ as probability measures on $W^n$. In fact, by the discussion after Theorem 1.3 of \cite{DeJeu}, since we work in $[0,\infty)^n$ and not the full space $\mathbb{R}^n$, we need not require that the symmetrized versions of these measures have exponential moments but that they only need to integrate $e^{\epsilon \sqrt{\|z\|}}$. The theorem is now proven.
\end{proof}

\bigskip
\noindent
{\sc Mathematics Institute, University of Warwick, Coventry CV4 7AL, U.K.}\newline
\href{mailto:T.Assiotis@warwick.ac.uk}{\small T.Assiotis@warwick.ac.uk}


\begin{thebibliography}{}

\bibitem{Anderson}
  {\sc G. W. Anderson ,}
\textit{A short proof of Selberg's generalized beta formula}, Forum Mathematicum, Vol. 3, 415-417,(1991).

 \bibitem{PhDThesis}
  {\sc T. Assiotis,}
\textit{PhD thesis at University of Warwick}, in preparation, (2017+).

 \bibitem{InterlacingDiffusions}
  {\sc T. Assiotis, N. O'Connell, J. Warren,}
\textit{Interlacing Diffusions}, Available from http://arxiv.org/abs/1607.07182, (2016).

\bibitem{BakerForrester}
  {\sc T. H. Baker, P. J. Forrester ,}
\textit{The Calogero-Sutherland model and generalized classical
polynomials}, Communications in Mathematical Physics, Vol. 188, 175-216,(1997)

\bibitem{BougerolJeulin}
  {\sc P. Bougerol, T. Jeulin,}
\textit{Paths in Weyl chambers and random matrices}, Probability Theory and Related Fields, Vol. 124, Issue 4, 517-543,(2002)


\bibitem{DemniJacobi}
  {\sc N. Demni,}
\textit{$\beta$-Jacobi processes}, Advances in Pure and Applied Mathematics, Vol. 1, 325-344, (2010).

 \bibitem{DemniBESQ}
  {\sc N. Demni,}
\textit{Radial Dunkl Processes : Existence and uniqueness, Hitting time, Beta Processes and Random Matrices}, Available from http://arxiv.org/abs/0707.0367, (2007).

\bibitem{Dixon}
  {\sc A. L. Dixon,}
\textit{Generalizations of Legendre's formula $KE'-(K-E)K'=\frac{1}{2}\pi$}, Proceedings of the London Mathematical Society, Vol. 3, 206-224, (1905).

\bibitem{LDPBessel}
  {\sc C. Donati-Martin, A. Rouault, M. Yor, M. Zani,}
\textit{Large deviations for squares of Bessel and
Ornstein-Uhlenbeck processes}, Probability Theory and Related Fields, Vol. 129, 261-289, (2004).

\bibitem{Fernique}
  {\sc X. Fernique,}
\textit{Integrabilite des vecteurs gaussiens},Comptes Rendus de l'Academie des Sciences Paris A-B, A1698 - A1699, (1970).

\bibitem{Forrester}
  {\sc P. J. Forrester ,}
\textit{Log-gases and random matrices}, Princeton University Press, (2010).


\bibitem{GorinShkolnikov}
  {\sc V. Gorin, M. Shkolnikov}
\textit{Multilevel Dyson Brownian motions via Jack polynomials},Probability Theory and Related Fields, Vol. 163, 413-463, (2015)

 \bibitem{Graczyk}
  {\sc P. Graczyk, J. Malecki,}
 \textit{Strong solutions of non-colliding particle systems}, Electronic Journal of Probability, Vol.19, 1-21, (2014).

\bibitem{DeJeu}
  {\sc M. de Jeu,}
\textit{Determinate multidimensional measures, the extended Carleman theorem and quasi-analytic weights}, Annals of Probability, Vol.31, No. 3, 1205-1227, (2003)

 \bibitem{O Connell}
  {\sc W. K\"onig, N. O'Connell,}
 \textit{Eigenvalues of the Laguerre Process as Non-Colliding Squared Bessel Processes}, Electronic Communications in Probability, Vol. 6, 107-114, (2001).
 
  \bibitem{O ConnellTams}
   {\sc N. O'Connell,}
  \textit{A path-transformation for random walks and the Robinson-Schensted correspondence}, Transactions of the American Mathematical Society, \textbf{355}, (2003).
  
   \bibitem{O ConnellYor}
    {\sc N. O'Connell, M. Yor,}
   \textit{A Representation for Non-Colliding Random Walks}, Electronic Communications in Probability, Vol. 7, 1-12, (2002).

\bibitem{OkounkovOlshanski}
  {\sc A. Okounkov, G. Olshanski ,}
\textit{Shifted Jack polynomials, binomial formula, and applications}, Mathematical Research Letters, Vol. 4, 69-78, (1997)


 \bibitem{PalShkolnikov}
  {\sc S. Pal, M. Shkolnikov,}
\textit{Intertwining diffusions and wave equations}, Available from https://arxiv.org/abs/1306.0857, (2015).

 \bibitem{RamananShkolnikov}
  {\sc K. Ramanan, M. Shkolnikov,}
\textit{Intertwinings of $\beta$-Dyson Brownian motions of different dimensions}, Available from https://arxiv.org/abs/1608.01597, (2016).

\bibitem{RevuzYor}
{\sc D. Revuz, M. Yor,}
\textit{Continuous Martingales and Brownian Motion}, Third Edition, A Series of Comprehensive Studies in Mathematics, Vol. 293, Springer-Verlag,(1999)


\bibitem{DunklProcesses}
{\sc M. Rosler, M. Voit,}
\textit{Markov processes related with Dunkl operators}, Advances in Applied Mathematics, Vol. 21, 575-643,(1998).


 \bibitem{Sun}
  {\sc Y. Sun,}
\textit{Laguerre and Jacobi analogues of the Warren process }, Available from https://arxiv.org/abs/1610.01635, (2016).


\bibitem{Walters}
{\sc P. Walters,}
\textit{An Introduction to Ergodic Theory}, Graduate Texts in Mathematics, Vol. 79, Springer-Verlag, (1982).

 \bibitem{Warren}
  {\sc J. Warren,}
 \textit{Dyson's Brownian motions,intertwining and interlacing}, Electronic Journal of Probability, Vol.12, 573-590, (2007).



\end{thebibliography}
\end{document}